\documentclass[final,nomarks]{dmtcs-episciences}
\usepackage[utf8]{inputenc}
\usepackage[T1]{fontenc}
\usepackage{enumitem}
\DeclareGraphicsExtensions{.pdf,.png,.eps}
\graphicspath{{figs/}}
\usepackage{amsmath,amssymb,amsthm}
\newtheorem{theorem}{Theorem}[section]
\newtheorem{lemma}[theorem]{Lemma}
\newtheorem{proposition}[theorem]{Proposition}
\newtheorem{corollary}[theorem]{Corollary}
\newtheorem{observation}[theorem]{Observation}
\usepackage{xspace}
\usepackage{mathtools}
\DeclarePairedDelimiter\abs{\lvert}{\rvert}
\DeclarePairedDelimiter\set{\{}{\}}
\newcommand{\PR}{\textup{PR}\xspace}
\newcommand{\PRZ}{\textup{PR$^0$}\xspace}
\newcommand{\PRP}{\textup{PR$^+$}\xspace}
\newcommand{\PRM}{\textup{PR$^-$}\xspace}
\DeclareMathOperator{\dPR}{\ensuremath{d_{\PR}}}
\newcommand{\TBR}{\textup{TBR}\xspace}
\newcommand{\TBRZ}{\textup{TBR$^0$}\xspace}
\newcommand{\TBRP}{\textup{TBR$^+$}\xspace}
\newcommand{\TBRM}{\textup{TBR$^-$}\xspace}
\DeclareMathOperator{\dTBR}{\ensuremath{d_{\TBR}}}
\DeclareMathOperator{\dAD}{\ensuremath{d_{\text{AD}}}}
\DeclareMathOperator{\dEAD}{\ensuremath{d_{\text{EAD}}}}
\def\distR{\ensuremath{d_R}}
\def\utrees{\ensuremath{u\mathcal{T}_n}}
\def\utreesx[#1]{\ensuremath{u\mathcal{T}_{#1}}}
\def\unets{\ensuremath{u\mathcal{N}_n}}
\def\unetsr{\ensuremath{u\mathcal{N}_{n,r}}}
\def\unetsx[#1]{\ensuremath{u\mathcal{N}_{#1}}}
\def\replugnets{\ensuremath{u\mathcal{M}_n}}

\usepackage[capitalise,noabbrev,nameinlink]{cleveref}

\usepackage[round]{natbib} 
\newcommand{\doi}[1]{\href{https://dx.doi.org/#1}{\texttt{doi:#1}}}

\author{Jonathan Klawitter}
\title[The agreement distance of unrooted phylogenetic networks]{The agreement distance of unrooted phylogenetic networks}
\affiliation{University of Würzburg, Germany}

\received{2019-8-24}
\revised{2020-1-17, 2020-5-22}
\accepted{2020-6-2}

\begin{document}
\publicationdetails{22}{2020}{1}{22}{5709}

\pdfbookmark[1]{Title}{title}  

\maketitle

\begin{abstract} 
\pdfbookmark[1]{Abstract}{Abstract} 
A rearrangement operation makes a small graph-theoretical change to a phylogenetic network to transform it into another one.
For unrooted phylogenetic trees and networks, popular rearrangement operations are tree bisection and reconnection (TBR) 
and prune and regraft (PR) (called subtree prune and regraft (SPR) on trees).
Each of these operations induces a metric on the sets of phylogenetic trees and networks. 
The TBR-distance between two unrooted phylogenetic trees $T$ and $T'$ can be characterised by a maximum agreement forest, that is,
a forest with a minimum number of components that covers both $T$ and $T'$ in a certain way. 
This characterisation has facilitated the development of fixed-parameter tractable algorithms and approximation algorithms. 
Here, we introduce maximum agreement graphs as a generalisations of maximum agreement forests for phylogenetic networks.
While the agreement distance -- the metric induced by maximum agreement graphs -- does not characterise the TBR-distance of two networks,
we show that it still provides constant-factor bounds on the TBR-distance.
We find similar results for PR in terms of maximum endpoint agreement graphs.
\end{abstract}

\section{Introduction}
\label{sec:introduction}

Phylogenetic trees and networks are graphs where the leaves are labelled bijectively by a set of taxa, 
for example a set of organisms, species, or languages~\citep{SS03,Dun14}.
They are used to model and visualise evolutionary relationships.
While a phylogenetic tree is suited only for tree-like evolutionary histories, 
a phylogenetic network can also be used for taxa whose past includes reticulate events like hybridisation, horizontal gene transfer,
recombination, or reassortment~\citep{SS03,HRS10,Ste16}. 
Such reticulate events arise in all domains of life~\citep{TN05,RW07,MMM17,WKH17}. 
There is a distinction between rooted and unrooted phylogenetic networks.
More precisely, in a rooted phylogenetic network the edges are directed from a designated root towards the leaves, 
thus modelling evolution along the passing of time. 
On the other hand, the edges of an unrooted phylogenetic network are undirected and the network thus represents the evolutionary relatedness of the taxa. 
In some cases, unrooted phylogenetic networks can be thought of as rooted phylogenetic networks 
in which the orientation of the edges has been disregarded.
\citet{JJEvIS17,FHM18,HIJJMMS19} call such unrooted phylogenetic networks proper.
Here we focus on unrooted, binary, proper phylogenetic networks where binary means that all vertices except for the leaves have degree three. 

A rearrangement operation makes a small graph-theoretical change to transform a phylogenetic network into another one.
Since this induces neighbourhoods, rearrangement operations structure the set of phylogenetic networks on the same taxa into a space.
Because of this property, they are used by several phylogenetic inference methods 
that traverse this space~\citep{Pag93,BEAST,MRBayes,PhyML,YBN13,YDLN14,WM15a}.
Furthermore, the minimum number of rearrangement operations needed to transform one network into another induces a metric.
This allows the comparison of results obtained for different data or from different inference methods, for instance, to evaluate
their robustness or to find outliers or clusters.

\begin{figure}[htb]
  \centering
  \includegraphics{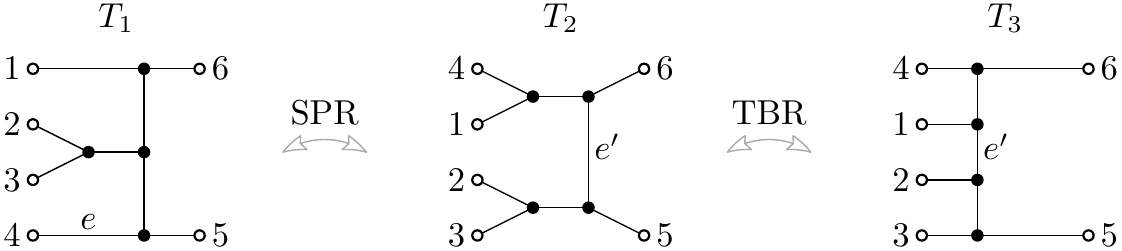}
  \caption{An SPR prunes the edge $e$ in $T_1$ and regrafts it to the edge incident to leaf 1 to obtain $T_2$. 
  		   A TBR moves the edge $e'$ to obtain $T_3$ from $T_2$.}
  \label{fig:treeIntro}
\end{figure}

On unrooted phylogenetic trees, two popular rearrangement operations are subtree prune and regraft (SPR),
which cuts (prunes) an edge at one side and then reattaches it,
and tree bisection and reconnection (TBR), which removes an edge and then reconnects the two resulting smaller trees~\citep{SOW96}. 
These are illustrated in \cref{fig:treeIntro}.
Computing the SPR- and TBR-distance of two unrooted phylogenetic trees $T$ and $T'$ is NP-hard~\citep{AS01,HDRB08}.
On the positive side, the TBR-distance of $T$ and $T'$ is characterised by a maximum agreement forest (MAF) of $T$ and $T'$,
which is a forest of smaller phylogenetic trees on which $T$ and $T'$ agree upon 
and that among all such forests has the minimum number of components~\citep{AS01}. 
The idea is that a MAF captures all parts that remain unchanged by a shortest TBR-sequence that transforms $T$ into $T'$.
\cref{fig:MAFintro} shows a maximum agreement forest $F$ for $T_1$ and $T_3$ of \cref{fig:treeIntro}.
Furthermore, a MAF $F$ together with the edges that got moved by the TBR-sequence can be embedded into $T$ and $T'$ 
such that all edges are covered; see again \cref{fig:MAFintro}.
Compared to a sequence of trees that describe a TBR-distance, MAFs provide a single structure
and have therefore been utilised for NP-hardness proofs~\citep{AS01,HDRB08}, 
for fixed-parameter tractable and approximation algorithms~\citep{AS01,HM07,RSW07,WZ09,CFS15}.
So far, no characterisation of the SPR-distance in terms of such a structure has been found
and \citet{WM18} argue why such a characterisation is unlikely.
In particular, they showed that an edge might be pruned twice and 
that common clusters (subtrees on a subset of the leaves) are not always maintained.
However, Whidden and Matsen introduced maximum endpoint agreement forests (MEAF) (precisely defined in \cref{sec:EAG}) 
as a variation of MAFs that bound the SPR-distance of two trees.

\begin{figure}[htb]
  \centering
  \includegraphics{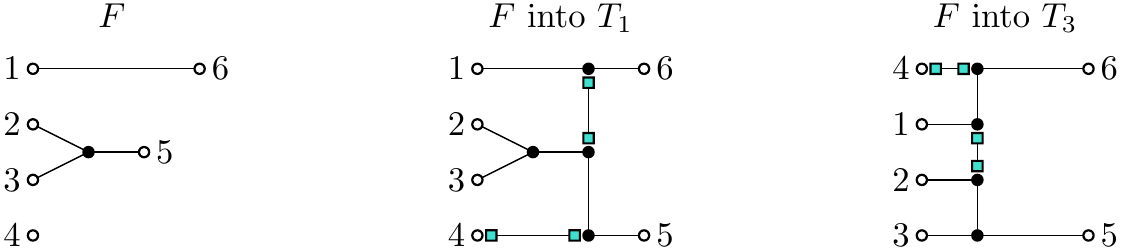}
  \caption{The graph $F$ is a maximum agreement forest for $T_1$ and $T_3$ of \cref{fig:treeIntro}. 
  Together with the edges moved by the operations, $F$ can be embedded into $T_1$ and $T_3$ such that all edges are covered.}
  \label{fig:MAFintro}
\end{figure}

SPR and TBR on trees have recently been generalised to phylogenetic networks with the operations prune and regraft (PR) and TBR~\citep{FHMW17,JK19}.
In principle, these operations work the same on networks as on trees. 
A PR operation first prunes an edge at one side and then reattaches it at another edge; 
a TBR operation on a network may also first remove an edge and then add a new edge like a TBR operation on a tree. 
This is illustrated in \cref{fig:networksIntro}.
However, both PR and TBR may also remove or add an edge to change the size of the network (see \cref{fig:unets:TBR}). 
\citet{JK19} studied several properties of spaces of networks under PR and TBR  
and, among other results, showed that computing the TBR-distance of two networks is NP-hard.

\begin{figure}[htb]
  \centering
  \includegraphics{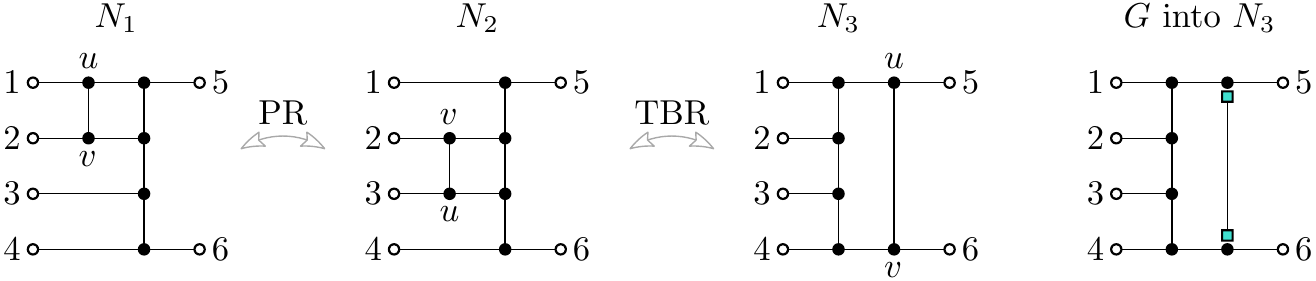}
  \caption{A PR prunes the edge $\set{u, v}$ at $u$ in $N_1$ and regrafts it to the edge incident to leaf 3 to obtain $N_2$. 
  		   A TBR moves the edge $\set{u, v}$ to obtain $N_3$ from $N_2$.
  		   The graph $G$ is a maximum agreement graph for $N_1$ and $N_3$ shown with an embedding into $N_3$.} 
  \label{fig:networksIntro}
\end{figure}

Similar to the TBR-distance of unrooted phylogenetic trees, 
the SPR-distance of two rooted phylogenetic trees can be characterised by a rooted version of MAFs~\citep{BS05}.  
This has again facilitated the development of fixed-parameter and approximation algorithms~\citep{BMS08,Wu09,BStJ09,WBZ13,BSTW17}.
Prune and regraft (PR) and subnet prune and regraft (SNPR) are extensions of SPR for rooted phylogenetic networks~\citep{BLS17,Kla19}. 
Recently, we generalised MAFs to maximum agreement graphs (MAG) for networks. 
Similar to a MAF, the idea of a MAG is that its components model those parts of two phylogenetic networks 
on which they agree upon and on which they disagree upon
(or the parts that stay unchanged and get changed under a sequence of rearrangement operations).
\cref{fig:networksIntro} illustrates this for the two unrooted phylogenetic networks $N_1$ and $N_3$. 
While MAFs characterise the SPR-distance, we have shown that MAGs do not characterise 
the PR-distance (nor the SNPR-distance) of two rooted phylogenetic networks~\citep{Kla19}.
This is due to similar reasons to why MAFs and MEAFs do not characterise the SPR-distance of two unrooted trees.
However, we showed that MAGs induce a metric on phylogenetic networks, the agreement distance, 
which bounds the PR-distance of rooted networks~\citep{Kla19}.

In this paper, we look at how MAF and MEAF generalise for unrooted phylogenetic networks 
by introducing maximum agreement graphs and maximum endpoint agreement graphs 
and show that they induce metrics (\cref{sec:AG} and \cref{sec:EAG}, respectively). 
We call these metrics the agreement distance (AD) and endpoint agreement distance (EAD), respectively.
We then study the relations of AD, EAD, the TBR-distance, and the PR-distance in \cref{sec:bounds}.

\section{Preliminaries}
\label{sec:preliminaries}

This section provides the notation and terminology used in the remainder of the paper.
In particular, we introduce notation in the context of phylogenetic networks as well
as the PR and TBR operations. 

\paragraph{Phylogenetic networks and trees.}
\pdfbookmark[2]{Phylogenetic networks}{PhyNets}
Let $X = \{1, 2, \ldots, n\}$ be a finite set.
An \emph{unrooted binary phylogenetic network} $N$ on $X$ is a connected undirected multigraph 
such that the leaves are bijectively labelled with $X$ and all non-leaf vertices have degree three.
It is called \emph{proper} if every cut-edge separates two labelled leaves \citep{FHM18}, and \emph{improper} otherwise.
Unless mentioned otherwise, we assume that a phylogenetic networks is proper.
Note that our definition permits the existence of parallel edges in~$N$.
An \emph{unrooted binary phylogenetic tree} on $X$ is an unrooted binary phylogenetic network that is a tree.
See \cref{fig:unets:treeAndNetwork} for examples.
An edge of $N$ is \emph{external} if it is incident to a leaf, and \emph{internal} otherwise.

To ease reading, we refer to a proper unrooted binary phylogenetic network (resp. unrooted binary phylogenetic
tree) on $X$ simply as a phylogenetic network or network (resp. phylogenetic tree or tree).
Furthermore, let $\unets$ denote the set of all phylogenetic networks on $X$ and let $\utrees$ denote
the set of all phylogenetic trees on $X$ where $n = \abs{ X }$.

\begin{figure}[htb]
  \centering
  \includegraphics{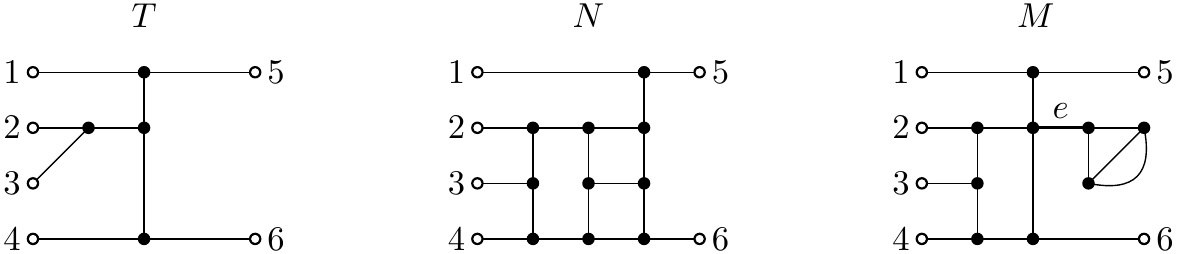}
  \caption{An unrooted, binary phylogenetic tree $T \in \utreesx[6]$ and an unrooted, binary proper phylogenetic network $N \in \unetsx[6]$. 
  The unrooted, binary phylogenetic network $M$ is improper since the cut-edge $e$ does not lie on a path that connects two leaves.}
  \label{fig:unets:treeAndNetwork}
\end{figure}

A network $N$ has \emph{reticulation number} $r$ or, equivalently, is in \emph{tier} $r$ 
if $r$ is the minimum number of edges that have to be deleted from $N$ to obtain a spanning tree of $N$.
Note that $r = \abs{E} - (\abs{V} - 1)$ where~$E$ and~$V$ are the edge and vertex set of $N$, respectively.
This number is also known as the cyclomatic number of a graph~\citep{Die17}.
Let $\unetsr$ denote tier $r$ of $\unets$, that is, the set of networks in $\unets$ that are in tier $r$.

\paragraph{Suboperations and sprouts.}
\pdfbookmark[2]{Suboperations}{suboperations}
Let $G$ be an undirected graph. 
A degree-two vertex $v$ of $G$ with adjacent vertices $u$ and $w$ gets \emph{suppressed} by deleting $v$ and its incident edges, and adding the edge $\set{u, w}$. 
The reverse of this suppression is the \emph{subdivision} of $\set{u, w}$ with a vertex $v$.

Let $\set{u, v}$ be an edge of $G$ such that $u$ either has degree one and is labelled (like a leaf of a network) or has degree three.
A \emph{pruning} of $\set{u, v}$ at $u$ is the process of deleting $\set{u, v}$ and adding a new edge~$\set{\bar u, v}$, 
where $\bar u$ is a new (unlabelled) vertex. If $u$ is now a degree two vertex, then we also suppress $u$. 
In reverse, the edge $\set{\bar u, v}$ gets \emph{regrafted} to an edge $\set{x, y}$ by subdividing $\set{x, y}$ with a new vertex $u$
and then identifying $\bar u$ and $u$.
Alternatively, $\set{\bar u, v}$ may be \emph{regrafted} to a labelled singleton $u$ by identifying $\bar u$ and~$u$.
The edge $\set{u, v}$ gets \emph{removed} by deleting $\set{u, v}$ from $N$ and suppressing any resulting degree-two vertices. 

A \emph{sprout} of $G$ is an unlabelled degree one vertex of $G$.
For example, applying a pruning to a phylogenetic network yields a graph with exactly one sprout.

\paragraph{Rearrangement operations.}
\pdfbookmark[2]{Rearrangement operations}{TBR} 
Let $N \in \unets$. 
The \TBR operation is the rearrangement operation that transforms $N$ 
into a phylogenetic network $N' \in \unets$ in one of the following four ways: 
\begin{itemize}[leftmargin=*,label=(TBR$^-$)]
    \item[(\TBRZ)] Remove an internal edge $e$ of $N$,  subdivide an edge of the resulting graph with a new vertex $u$, 
    subdivide an edge of the resulting graph with a new vertex $v$, and add the edge $\set{u, v}$; or\\
    prune an external edge $e = \set{u, v}$ of $N$ that is incident to leaf $v$ at $u$, regraft the resulting sprout to an edge of the resulting graph.
    \item[(\TBRP)] Subdivide an edge of $N$ with a new vertex $u$, subdivide an edge of the resulting graph with a new vertex $v$, and add the edge $e = \set{u, v}$.
    \item[(\TBRM)] Remove an edge $e$ of $N$.
\end{itemize}
Note that a \TBRZ can also be seen as the operation that prunes the edge $e = \set{u, v}$ at both $u$ and $v$ and then regrafts the two resulting sprouts. 
Hence, we say that a \TBRZ \emph{moves} the edge $e$. Furthermore, we say that a \TBRP \emph{adds} the edge $e$ and that a \TBRM \emph{removes} the edge $e$. 
TBR is illustrated in \cref{fig:unets:TBR}. 
Note that a \TBRZ has an inverse \TBRZ and that a \TBRP has an inverse \TBRM, 
and that furthermore a~\TBRP increases the reticulation number by one and a \TBRM decreases it by one.
On trees, \TBRZ equals the well known tree bisection and reconnection operation~\citep{AS01}, which is also where the acronym comes from. 

\begin{figure}[htb]
  \centering
  \includegraphics{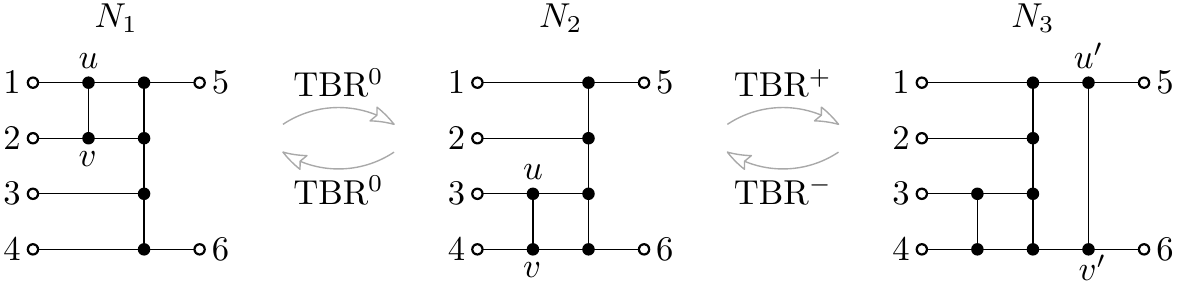}
  \caption{Illustration of the TBR operation. 
  The network $N_2$ can be obtained from $N_1$ by a \TBRZ that moves the edge $\set{u, v}$ 
  and the network $N_3$ can be obtained from $N_2$ by a \TBRP that adds the edge $\set{u', v'}$. 
  Each operation has its corresponding \TBRZ and \TBRM operation, respectively, that reverses the rearrangement.} 
  \label{fig:unets:TBR} 
\end{figure}

Since a \TBR operation has to yield a phylogenetic network, there are some restrictions on the edges that can be moved or removed. 
Firstly, if removing an edge by a \TBRZ yields a disconnected graph, then in order to obtain a phylogenetic network an edge has to be added between the two connected components. 
For similar reasons, a \TBRM cannot remove a cut-edge. 
Secondly, the suppression of a vertex when removing an edge with a \TBRM may not yield a loop $\set{u, u}$.  
Thirdly, removing or moving an edge cannot create a cut-edge that does not separate two leaves. Otherwise the resulting network would be improper.

Let $N \in \unets$. 
A \PR (\emph{prune and regraft}) operation is the rearrangement operation that transforms $N$ into a phylogenetic network $N' \in \unets$ with a \PRP $=$ \TBRP, a \PRM $=$ \TBRM, 
or a \PRZ that prunes and regrafts an edge $e$ only at one endpoint, instead of at both like a \TBRZ~\citep{JK19}. 
Like for TBR, we the say that the PR$^{0/+/-}$ \emph{moves/adds/removes} the edge $e$ in $N$. 
The PR operation is a generalisation of the well known SPR (\emph{subtree prune and regraft}) operation on unrooted phylogenetic trees~\citep{AS01}. 

\paragraph{Distances.}
\pdfbookmark[2]{Distances}{distances} 
Let $N, N' \in \unets$.
A \emph{\TBR-sequence} from $N$ to $N'$ is a sequence
    $$\sigma = (N = N_0, N_1, N_2, \ldots, N_k = N') $$
of phylogenetic networks such that $N_i$ can be obtained from $N_{i-1}$ by a single TBR for each $i \in \set{1, 2, ..., k}$. 
The \emph{length} of $\sigma$ is $k$.
The \emph{\TBR-distance} $\dTBR(N, N')$ between $N$ and $N'$ is the length of a shortest TBR-sequence from $N$ to $N'$.
The \PR-distance is defined analogously. 
\citet[Corollary 4.4]{JK19} have shown that the TBR- and PR-distance are well defined.

\paragraph{Embeddings and displaying.}
\pdfbookmark[2]{Embeddings}{embedding} 
Let $G$ be an undirected graph that is not necessarily simple; that is, $G$ may contain parallel edges and loops. 
An edge $\set{u, v}$ of $G$ is \emph{subdivided} if $\set{u, v}$ is replaced by a path form $u$ to $v$ that contains at least one edge. 
A \emph{subdivision} $G^*$ of $G$ is a graph that can be obtained from $G$ by subdividing edges of $G$. 
If $G$ has no degree two vertices, there exists a canonical mapping of vertices of $G$ to vertices of $G^*$ and of edges of $G$ to paths of $G^*$.

Let $N$ be an undirected graph, for example a network in $\unets$. 
Assume that $G$ is connected. 
We say~$G$ has an \emph{embedding} into $N$ if there exists a subdivision $G^*$ of $G$ that is a subgraph of $N$.
Now assume that $G$ has components $C_1, \ldots, C_k$. 
We say $G$ has an \emph{embedding} into $N$ if the components $C_i$ of $G$, for $i \in \{1, \ldots, k\}$, have embeddings into $N$
to pairwise edge-disjoint subgraphs of $N$. 
Note that these definitions imply that a labelled vertex of $G^*$ is mapped to a labelled vertex of $N$ with the same label.

We define a special type of embedding. % Recall that a trail in a graph is a walk without repeated edges.
Let $n$ vertices of $G$ be labelled bijectively with $X = \{1, 2, \ldots, n\}$. 
We say $G$ has an \emph{agreement embedding} into $N$ if there exists an embedding of $G$ into $N$ with the following properties.
\begin{itemize}
    \item The pairwise edge-disjoint embeddings of components of $G$ into $N$ cover all edges.
    \item At most two vertices of $G$ are mapped to the same vertex of $N$. 
    In the case that exactly two vertices are mapped to the same vertex of $N$, 
    one of these two vertices of $G$ is a sprout and the other is a labelled, isolated vertex.
    \item For each labelled vertex $v$ of $N$, there exists exactly one vertex $\bar v$ with the same label in $G$ and $\bar v$ is mapped to $v$.
\end{itemize}
We make the observation that having an agreement embedding into a graph is a transitive property.
\begin{observation} \label{clm:unets:AE:transitive}
Let $G$, $H$, $N$ be undirected graphs such that $G$ has an agreement embedding into $H$ and $H$ has an agreement embedding into $N$.\\
Then $G$ has an agreement embedding into $N$.
\end{observation}

Let $N, N' \in \unets$.
We say $N'$ \emph{displays} $N$ if $N$ has an embedding into $N'$.
For example, in \cref{fig:unets:treeAndNetwork} the tree $T$ is displayed by both networks $N$ and $M$.

\section{Agreement graph and distance}
\label{sec:AG}

In this section we look at how agreement forests can be generalised for networks.
Throughout this section, let $N, N' \in \unets$ be in tier $r$ and $r'$, respectively.
Without loss of generality, assume that $r' \geq r$ and let~$l = r' - r$. 

% The idea is as follows. 
Suppose there is a \TBRZ that transforms $N$ into $N'$ by moving an edge $e$.
This operation can be seen as removing $e$ from $N$, obtaining a graph $S$,
and then adding a new edge to $S$. 
We can interpret $S$ as the part of $N$ that remains unchanged or, in other words, $N$ and $N'$ agree on $S$.
In general, we are interested in finding a graph that requires the minimal number of edge removals from $N$ (or $N'$)
such that it has an embedding into $N$ and $N'$.
For two trees $T$ and $T'$ in $\unets$, this graph is precisely a \emph{maximum agreement forest (MAF)} $F$.
\citet{AS01} showed that the number of components of $F$ minus one
is exactly the \TBR-distance of $T$ and $T'$, 
or, equivalently, the minimum number of edges that have to be removed from $T$ (or $T'$) to obtain $F$.
If we consider again $N$ and $N'$, then the removal of an edge must not necessarily increase the number of components.
Therefore, instead of counting components, we are looking for a graph $G$ consisting of components on which $N$ and $N'$
agree on and of additional edges that can be embedded into $N$ and $N'$ such that all edges are covered.
In other words, we want that $G$ has an agreement embedding into $N$ and $N'$.
Note that if $N$ and $N'$ are in different tiers, then we need additional edges for an agreement embedding into $N'$.
We now make this precise. 

\paragraph{Agreement graph.}
Let $G$ be an undirected graph with connected components $S_1, \ldots, S_m$ and \linebreak[4]$E_1, \ldots, E_{k - l}, E_{k -l +1}, \ldots, E_{k}$
such that the $S_i$'s contain no sprouts and 
such that each $E_j$ consist of a single edge on two unlabelled vertices.
Then $G$ is an \emph{agreement graph} of $N$ and $N'$ if 
\begin{itemize}
    \item $G$ without $E_{k - l + 1}, \ldots, E_{k}$ has an agreement embedding into $N$, and
    \item $G$ has an agreement embedding into $N'$.
\end{itemize}
For such an agreement graph, we refer to an $S_i$ as \emph{agreement subgraph} and to an $E_j$ as a \emph{disagreement edge}.
A \emph{maximum agreement graph (MAG)} $G$ of $N$ and $N'$ is an agreement graph of $N$ and $N'$ with a minimal number of disagreement edges.
See \cref{fig:unets:AGexample1,fig:unets:AGexample2} for two examples.

\begin{figure}[htb]
  \centering
  \includegraphics{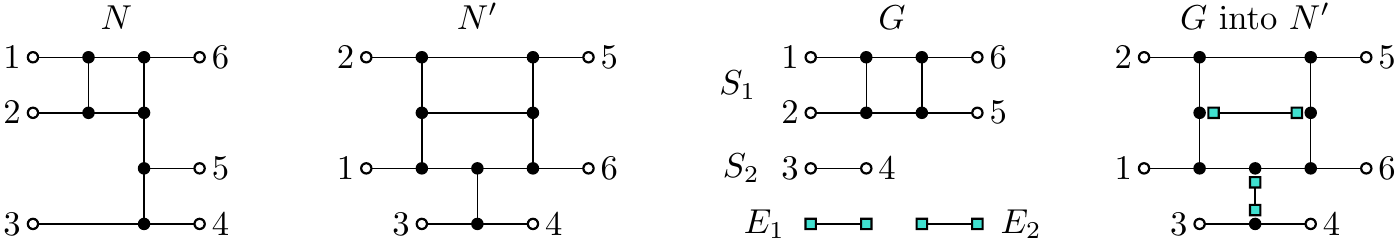} 
  \caption{A maximum agreement graph $G$ for $N, N' \in \unets$. On the right, how $G$ embeds into $N'$.
  Note that the disagreement edge $E_2$ is only needed for an agreement embedding into $N'$.}
  \label{fig:unets:AGexample1}
\end{figure}

\begin{figure}[htb]
  \centering
  \includegraphics{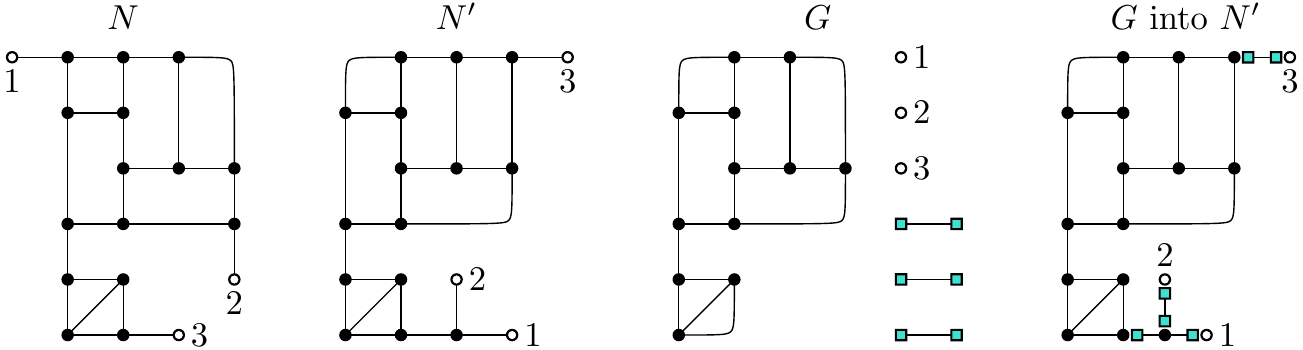} 
  \caption{A maximum agreement graph $G$ for $N, N' \in \unets$. On the right, how $G$ embeds into $N'$.
  Note that $G$ contains an agreement subgraph without labelled vertices.} 
  \label{fig:unets:AGexample2}
\end{figure}

Note that if $G$ contains $m$ agreement subgraphs, then it also contains at least $m-1$ disagreement edges since $N$ and $N'$ are connected graphs. 
Furthermore, unlike a MAF for two phylogenetic trees, $G$ may contain agreement subgraphs without any labelled vertices and $G$ may contain loops or parallel edges.

Let $T, T' \in \utrees$. Let $G$ be a maximum agreement graph of $T$ and $T'$.
Note that each agreement subgraph of $G$ is a tree.
Thus a (maximum) agreement graph of two trees is called a (maximum) agreement forest.
Note that the commonly used definition of agreement forests disregards disagreement edges~\citep{AS01}.

\paragraph{Attached sprouts.}
Let $V_N$ and $E_N$ be the vertex and edge set of $N$, respectively.
Let $G = (V_G, E_G)$ be an agreement graph of $N$ and $N'$.
Fix an agreement embedding of $G$ into $N$.
We say a sprout $\bar u \in V_G$ is \emph{attached to $\bar e \in E_G$ in $N$} 
if $\bar u$ is mapped to a vertex $u \in V_N$ that is an internal vertex of the path to which~$\bar e$ is mapped.
Suppose $G$ contains a labelled singleton $\bar x$.
We say $\bar u \in V_G$  is \emph{attached to $\bar x$ in $N$} if $\bar u$ and~$\bar x$ are mapped to the same leaf $u \in V_N$.
This terminology can be extended from sprouts to disagreement edges.
We say a disagreement edge $E_i$ is \emph{attached to an edge $\bar e \in E_G$ in $N$} if a sprout of $E_i$ is attached to $\bar e$ in $N$. 
Furthermore, we say $E_i$ is \emph{attached to} an agreement subgraph $S_{j}$ \emph{in $N$} 
if $E_i$ is attached to an edge of~$S_{j}$.

Note that a disagreement edge can be attached to itself. 
However, in general we would like to assume that an agreement embedding has nicer properties.
This is what we look at next.

\paragraph{Ordered agreement embedding.}
Let $G$ be an agreement graph of $N$ and $N'$ with agreement subgraphs $S_1, \ldots, S_m$ and disagreement edges $E_1, \ldots E_{k}$.
Then an agreement embedding of $G$ into $N'$ is an \emph{ordered agreement embedding} into $N'$ if
\begin{itemize}
    \item $E_1$ is attached to two distinct agreement subgraphs in $N'$,
    \item $E_i$ for $i \in \set{2, \ldots, m-1}$ is attached to 
        two distinct agreement subgraphs or an agreement subgraph and a disagreement edge $E_j$ with $j < i$ in $N'$
        such that the subgraph of $N'$ covered by $S_1, \ldots, S_m$ and $E_1, \ldots, E_i$ contains one connected component less than
        the subgraph of $N'$ covered by $S_1, \ldots, S_m$ and $E_1, \ldots, E_{i-1}$, 
    \item $E_i$ for $i \in \set{m, \ldots, k}$ is attached to agreement subgraphs or disagreement edges $E_j$ with $j < i$ in~$N'$.
\end{itemize}
An ordered agreement embedding of $G$ into $N$ is defined analogously but with the small difference 
that the third property only concerns the edges $E_{m+1}, \ldots, E_{k-l}$.

Note that the first and second property of an ordered agreement embedding imply that the vertices and edges covered by the agreement subgraphs 
and the disagreement edges $E_1, \ldots, E_{m-1}$ form a connected subgraph of $N$.
Moreover, in an ordered agreement embedding no disagreement edge is attached to itself.
We now prove that an agreement graph always has an ordered agreement embedding.

\begin{lemma}\label{clm:unets:orderedAE}
Let $N, N' \in \unets$ be in tiers $r$ and $r' \geq r$, respectively. Let $l = r' - r$.
Let $G$ be a maximum agreement graph for $N$ and $N'$ with $m$ agreement subgraphs.\\
Then $G$ minus $l$ disagreement edges has an ordered agreement embedding into $N$ 
and $G$ has an ordered agreement embedding into $N'$.
\end{lemma}
\begin{proof}
The proof works the same for $N$ and $N'$, so for simplicity we may assume that $l = 0$.
Since $G$ is a MAG of $N$ and $N'$, there is an agreement embedding $\phi$ of $G$ into $N$. 
Let $N_i$, $i \in \set{1, \ldots, m}$, be the subgraphs of $N$ to which the agreement subgraphs of $G$ are mapped by $\phi$. 
Colour all vertices and edges contained in these $N_i$'s black, and all other vertices and edges red.
The red edges are thus the edges to which the disagreement edges of $G$ are mapped.
Note that the $N_i$'s are vertex-disjoint.
Hence, since $N$ is connected, it follows that the $N_i$'s are connected by red edges and paths.
We use this fact to construct an ordered agreement embedding $\phi'$ of $G$ into $N$.

For the ordered agreement embedding $\phi'$ map the agreement subgraphs of $G$ into $N$ like $\phi$.
Pick $N_i$ and~$N_j$ such that there is path $P$ from $N_i$ to $N_j$ with black end vertices and with red internal vertices and edges.
Such a choice is possible by the observations above.
Let $\phi'$ map $E_1$ to $P$.
Colour the edges and vertices of $P$ black, which makes $N_i$ and $N_j$ a single black subgraph $N_i$.
Repeat this process for $E_2, \ldots, E_{m-1}$. Note that this results in a single black component in $N$.
Hence, for the remaining disagreement edges $E_{m}, \ldots, E_k$ we require from $P$ only that it contains black end vertices and red internal vertices and edges,
but not that $P$ connects two distinct black components.
As long as there remain red edges, we can find such $P$ with a simple depth-first search in a red component that starts at a red edge incident to a black vertex and
ends at a red edge incident to another black vertex. (Note that a red vertex always has degree three and a black, non-leaf vertex has at least degree two.) 
Therefore this process ends with all edges of $N$ covered and coloured black.
From a simple counting argument we get that we constructed exactly as many disagreement edges as $G$ has.
Hence, by construction the embedding $\phi'$ is an ordered agreement embedding of $G$ into $N$. 
\end{proof}

We now define how to change an agreement embedding gradually. 
Let $G$ be a MAG of $N$ and $N'$.
Let~$\bar u$ and~$\bar v$ be two sprouts of $G$ with incident edges $\bar e = (\bar u, \bar w)$ 
and $\bar f = (\bar v, \bar z)$, respectively, such that $\bar u$ is attached to $\bar f$ in $N$.
Let $\bar e$ be mapped to the path $P = (y, \ldots, w)$ in $N$ and let $\bar f$ be mapped to the
path $P' = (x, \ldots, y, \ldots, z)$ in $N$.
Then an \emph{embedding change} of $G$ into $N$ \emph{with respect to $\bar u$ and $\bar v$} is the change
of the embedding such that $\bar e$ is mapped to the path $(x, \ldots, y, \ldots, w)$ formed by a
subpath of $P'$ and the path $P$, and such that $\bar f$ is mapped to the subpath $(y, \ldots, z)$ of $P'$; see \cref{fig:AD:embeddingChange}. 

\begin{figure}[htb]
\centering
  \includegraphics{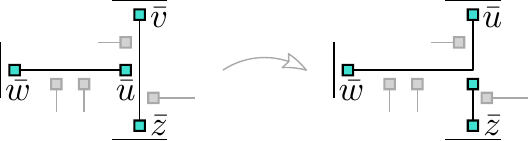}
  \caption{An embedding change with respect to $\bar u$ and $\bar v$.} 
  \label{fig:AD:embeddingChange}
\end{figure}

\paragraph{Agreement distance.}
Let $N, N' \in \unets$.
Let $G$ be a MAG of $N$ and $N'$ with $k$ disagreement edges.
We define the \emph{agreement distance} $\dAD$ of $N$ and $N'$ as 
    $$\dAD(N, N') = k \text{.}$$
Note that the agreement distance also equals half the number of sprouts of $G$.

\begin{theorem}\label{clm:AD:metric}
The agreement distance $\dAD$ on $\unets$ is a metric.
\end{theorem}
  \begin{proof}
    Note that $\dAD$ is symmetric, non-negative, and for all $M, M' \in \unets$ $\dAD(M, M') = 0$ if and only if $M = M'$.
    Therefore, to show that $\dAD$ is a metric, it remains to show that $\dAD$ satisfies the triangle inequality.
	
	Let $N, N', N'' \in \unets$ and let $k' = \dAD(N, N')$ and $k'' = \dAD(N', N'')$.
	Let $G'$ (resp. $G''$) be a MAG of $N$ and $N'$ (resp. $N'$ and $N''$) with $k'$ (resp. $k''$) disagreement edges.
	To show that the triangle inequality holds, we construct an agreement graph $G$ for $N$ and $N''$ with at most $k' + k''$ disagreement edges.
	The following construction is illustrated with an example in \cref{fig:unets:ADmetricExample}.
	
	For simplicity, assume for now that $N$, $N'$, and $N''$ are in the same tier. 
	Fix ordered agreement embeddings of $G'$ and $G''$ into $N'$, which is possible by \cref{clm:unets:orderedAE}.
	Based on the ordered agreement embedding we can construct a length-$k''$ sequence of graphs $(N' = M_0, M_1, \ldots, M_{k''} = G'')$ from $N'$ to $G''$ where
	$M_i$ is obtained from $M_{i-1}$ for $i \in \set{1, \ldots, k''}$ by removing an edge and adding a disagreement edge.
	Note that $M_{i}$ has an agreement embedding into $M_{i-1}$ and thus by the transitive property of agreement embeddings (recall \cref{clm:unets:AE:transitive}) also an agreement embedding into $N'$.
	We use this sequence, to construct a sequence of graphs $(G' = G_0, G_1, \ldots, G_{k''} = G)$ 
	such that $G_i$ is obtained from~$G_{i-1}$ for~$i \in \set{1, \ldots, k''}$
	either by setting $G_i = G_{i-1}$ or by the removal of an edge of an agreement subgraph and adding a disagreement edge.
	First, colour the disagreement edges of $G_0$ red. 
	We will colour each newly added disagreement edge blue.
	Our construction will ensure the following properties:
	\begin{itemize}
	  \item The only sprouts of $G_i$ are in disagreement edges;
	  \item $G_i$ has an agreement embedding into $M_i$, 
	  \item each blue disagreement edge of $G_i$ is mapped to a disagreement edge of $M_i$.  
	\end{itemize}
% 	We ensure that each $G_i$ has an agreement embedding into $M_i$. %, $N'$, and $N$.
% 	Moreover, our construction will ensure that the only sprouts of $G_i$ are in disagreement edges.
% 	First, colour the disagreement edges of $G_0$ red. 
% 	We will colour each newly added disagreement edge blue (and map them to disagreement edges of $M_i$).

	Suppose from $M_{i-1}$ to $M_{i}$ the edge $e$ gets removed and disagreement edge $F$ added. 
	We distinguish three cases, which are illustrated in \cref{fig:unets:ADmetric}~(a) to (c).
	First, if there is an edge $\bar e$ of $G_{i-1}$ that is mapped to $e$ by the agreement embedding of $G_{i-1}$ into $M_{i-1}$ 
	and that is not incident to a sprout, then obtain $G_i$ from $G_{i-1}$ by removing $\bar e$ and adding a blue disagreement edge $E_j$.
	This is shown in \cref{fig:unets:ADmetric}~(a) and in the step from $G_0$ to $G_1$ in \cref{fig:unets:ADmetricExample}.
	Note that $G_i$ has an agreement embedding in $M_i$ where $E_j$ is mapped to $F$.
	Clearly $G_i$ also has an agreement embedding into $G_{i-1}$ and thus by \cref{clm:unets:AE:transitive} into $N$ and $N'$.
	
	Second, suppose that an edge $\bar e$ of an agreement subgraph of $G_{i-1}$ is mapped to a path $P_e$ of $M_{i-1}$ that contains $e = \set{u, v}$.
	If $u$ (or $v$) lies within $P_e$, then a sprout of a disagreement edge of $G_{i-1}$ is attached to it.
	Note that this sprout belongs to a red disagreement edge since blue disagreement edges are mapped to edges that got removed in an earlier step.
	Obtain $G_i$ from $G_{i-1}$ by removing $\bar e$ and adding a blue disagreement edge $E_j$. 
	This case also applies in the step from $G_1$ to $G_2$ in \cref{fig:unets:ADmetricExample}.
	For the agreement embedding of $G_i$ into $M_i$ apply an embedding change (or embedding changes) as shown in \cref{fig:unets:ADmetric}~(b).
	Then $E_j$ is mapped onto $F$. 
    	
  \begin{figure}[htb]
    \centering
	\includegraphics{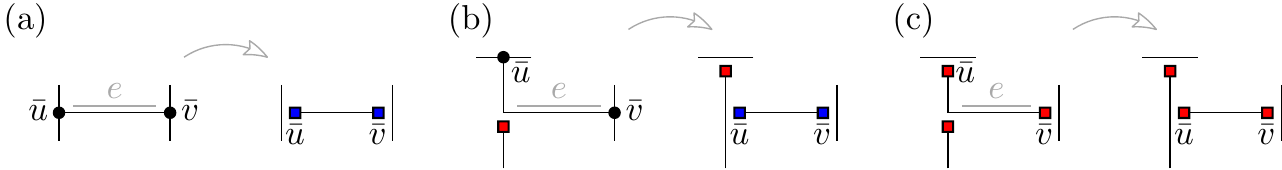}
    \caption{How to obtain $G_i$ (right) from $G_{i-1}$ (left) with respect to the agreement embedding of $G_{i-1}$ into $M_{i-1}$ 
    when an edge of an agreement subgraph is mapped precisely to $e$ (a) or to a path containing $e$ (b); 
    or when an disagreement edge is mapped to a path containing $e$ (c). 
    It is also shown how embedding changes are applied to show that $G_i$ has an agreement embedding into $M_i$.}
    \label{fig:unets:ADmetric}
  \end{figure}
  
  \begin{figure}[htb]   
    \centering
	\includegraphics{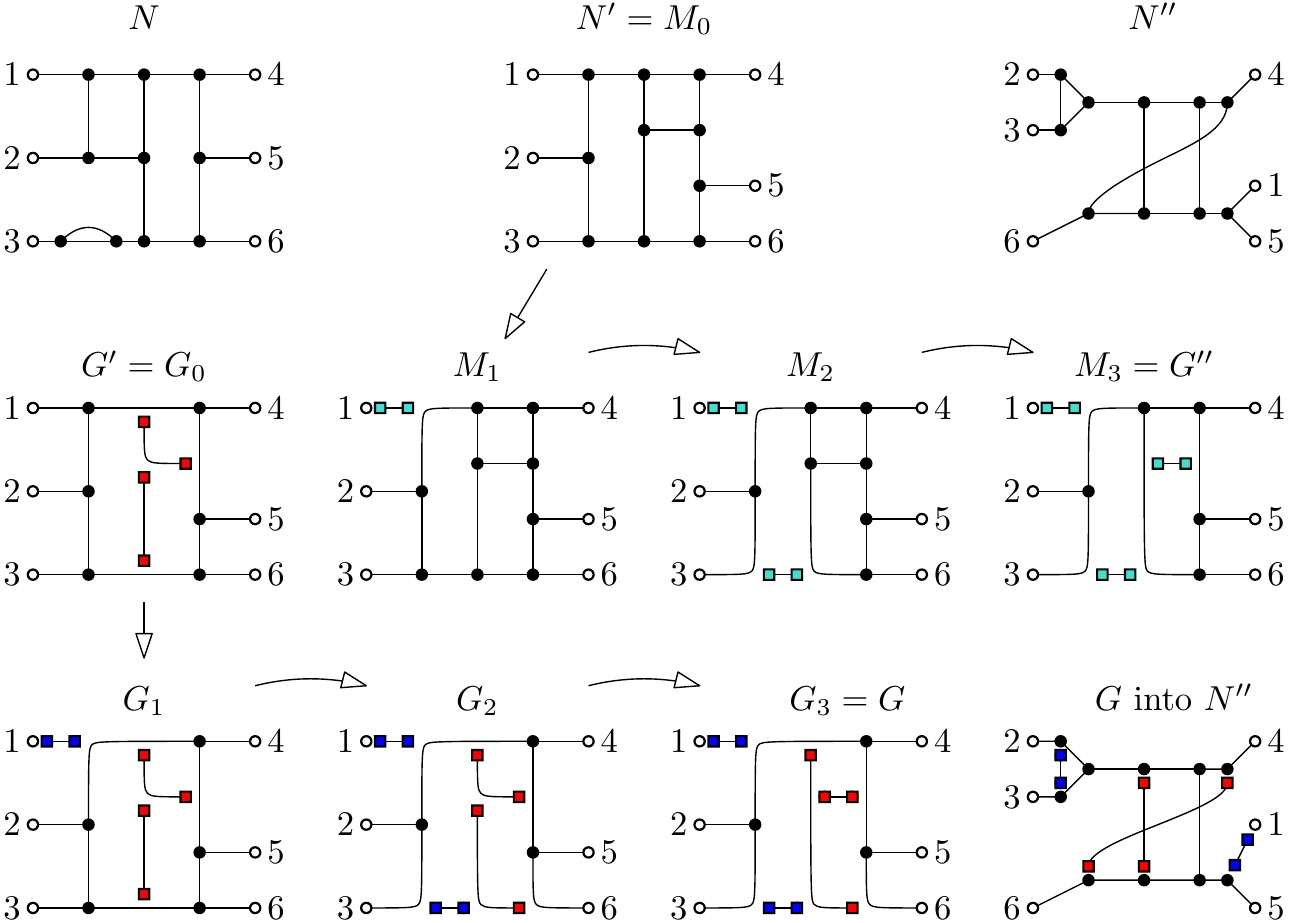}
    \caption{An example for the constructions of the sequences $(N' = M_0, M_1, \ldots, M_{k''} = G'')$ and 
    $(G' = G_0, G_1, \ldots, G_{k''} = G)$. Here, the networks $N$ and $N'$ have agreement distance 2 
    with maximum agreement graph $G'$;
    the networks $N'$ and $N''$ have agreement distance 3 with maximum agreement graph $G''$.
    Lastly, an agreement embedding of the constructed $G$ into $N''$ is shown.}
    \label{fig:unets:ADmetricExample}
  \end{figure}
	
	Third, suppose that a red disagreement edge $\bar e$ of $G_{i-1}$ is mapped to a path $P_e$ that contains $e = \set{u, v}$. 
	In this case set $G_i = G_{i-1}$. 
	To obtain an agreement embedding of $G_i$ into $M_i$ apply again appropriate embedding changes as shown in \cref{fig:unets:ADmetric}~(c)
	and in the step from $G_2$ to $G_3$ in \cref{fig:unets:ADmetricExample}.
% 	The agreement embedding of $G_i$ into $N$ is the same as that of $G_{i-1}$.
	
	We claim that $G = G_{k''}$ is an agreement graph of $N$ and $N''$.
	By \cref{clm:unets:AE:transitive} $G$ has an agreement embedding into $G'$ and thus into $N$ (and $N'$).
	Furthermore, $G$ has an agreement embedding into~$M_{k''} =~G''$. Therefore, again by \cref{clm:unets:AE:transitive},
	we get that $G$ has an agreement embedding into $N''$.
% 	Consider the agreement embedding of $G$ into $N'$. 
% 	If we fix any red disagreement edges to where they are attached to, then red disagreement edges and agreement subgraphs of $G$
% 	are mapped precisely to where agreement subgraphs of $G''$ are mapped to in $N'$. 
% 	Moreover, the blue disagreement edges of $G$ are mapped precisely to where disagreement edges of $G''$ are mapped to in $N'$.
% 	Hence, we can derive an agreement embedding of $G$ into $N''$ by the agreement embedding of $G''$ into $N''$. 
	Concerning the components of $G$, note that $G$ contains precisely $k'$ red disagreement edges and at most $k''$ blue disagreement edges. 
	By construction these disagreement edges contain all sprouts of $G$. Hence, $G$ is an agreement graph of $N$ and $N''$ 
	that proves that $\dAD(N, N'') \leq k' + k'' = \dAD(N, N') + \dAD(N', N'')$.
	This concludes the proof for the case when $N$, $N'$, and $N''$ are in the same tier.
	
	Cases where $N$, $N'$, and $N''$ are in different tiers work analogously.
	However, when for example $N''$ is in a higher tier than $N'$ then in the construction of the sequence $(N' = M_0, M_1, \ldots, M_{k''} = G'')$ from 
	$N'$ to $G''$ we stop removing edges at some point and only add disagreement edges. These extra disagreement edges are only needed for the agreement embedding into $N''$ but not $N'$.
	The same applies then to the construction of $(G' = G_0, G_1, \ldots, G_{k''} = G)$. 
  \end{proof}

Next, we show that if we restrict the agreement distance to the space of phylogenetic trees, then it
equals the TBR-distance. 

\begin{proposition}\label{clm:AD:treeRestriction}
The agreement distance is equivalent to the \TBR-distance on $\utrees$.
\end{proposition}
  \begin{proof}
	Let $G$ be a maximum agreement forest of two trees $T, T' \in \utrees$.
    \citet{AS01} defined the function $m(T, T')$ as the number of agreement subgraphs of $G$ minus one.
	If $G$ contains $k$ disagreement edges, then it contains $k + 1$ agreement subgraphs.
	Thus, $\dAD(T, T') = m(T, T')$. By Theorem~2.13 of \citet{AS01}, $m(T, T') = \dTBR(T, T')$.
	This concludes the proof.
  \end{proof}

\citet{AS01} further showed that computing the TBR-distance of two phylogenetic trees is NP-hard.
\citet[Theorem 6.1]{JK19} showed that the TBR-distance of two trees in $\utrees$ is the same as in $\unets$. 
These two results together with \cref{clm:AD:treeRestriction} give us the following corollary.

\begin{corollary}\label{clm:MAF:NPhard}
Computing the agreement distance on $\unets$ is NP-hard.
\end{corollary}

\section{Endpoint agreement graph and distance}
\label{sec:EAG}

While a \TBRZ prunes an edge at both ends, a \PRZ only prunes an edge at one side.
Hence, agreement graphs are not suited to model \PRZ. 
In this section we introduce endpoint agreement graphs as a slight modification of agreement graphs 
which model \PRZ more closely.
Let again $N, N' \in \unets$ be in tiers $r$ and $r'$, respectively, and let $l = r' - r$.

\paragraph{Endpoint agreement graph.}
Let $H$ be an undirected graph with connected components $S_1, \ldots, S_m$ and $E_1, \ldots, E_{l}$
such that each $E_j$ consist of a single edge on two unlabelled vertices.
Then $H$ is an \emph{endpoint agreement graph} (EAG) of $N$ and $N'$ if 
\begin{itemize}
    \item $H$ without $E_{1}, \ldots, E_{l}$ has an agreement embedding into $N$, and
    \item $H$ has an agreement embedding into $N'$.
\end{itemize}

We refer to an $S_i$ as \emph{(endpoint) agreement subgraph} and to an $E_j$ as a \emph{disagreement edge}.
A \emph{maximum endpoint agreement graph (MEAG)} $H$ of $N$ and $N'$ is an endpoint agreement graph of $N$ and $N'$ with a minimal number of sprouts.
See \cref{fig:unets:EAGexample1} for an example. Note that, unlike to MAG, in a MEAG also endpoint agreement subgraphs can contain sprouts. 
We define an \emph{ordered agreement embedding} of $H$ into $N'$ as an agreement embedding of $H$ into $N'$ such that
\begin{itemize}
  \item no sprout of an agreement subgraph is attached to a disagreement edge and
  \item and the disagreement edges can be ordered $(E_1, \ldots, E_l)$ such that $E_j$ may be attached to $E_i$ only if $i \leq j$.
\end{itemize}
For an ordered agreement embedding of $H$ into $N$ only the first property has to hold.

\begin{figure}[htb]
  \centering
  \includegraphics{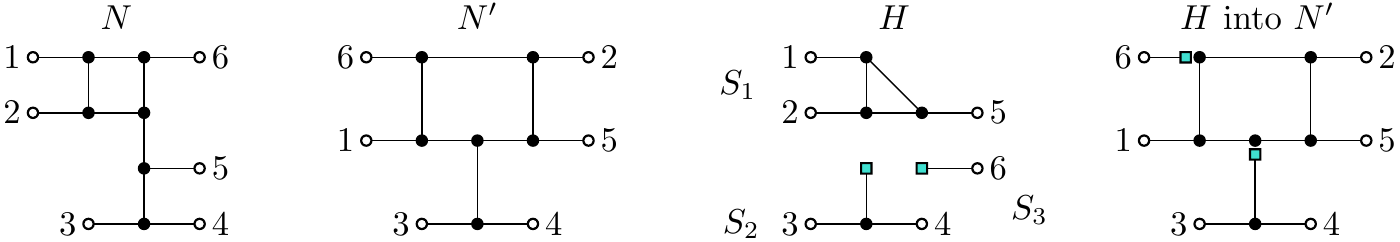} 
  \caption{A maximum endpoint agreement graph $H$ for $N, N' \in \unets$. On the right, how $H$ embeds into $N'$.} 
  \label{fig:unets:EAGexample1}
\end{figure}

A proof that ordered agreement embeddings exists works analogously to the proof of \cref{clm:unets:orderedAE} 
and the proof of Lemma 3.2 \citep{Kla19}, 
yet we outline the proof idea here. Starting with an agreement embedding of $G$ into $N'$, 
apply embedding changes to any sprout of an endpoint agreement subgraph that is attached to a disagreement edge. 
This way the first property can be enforced. 
For the second property, apply embedding changes if $E_i$ that is attached to a disagreement edge $E_j$, $j > i$, for $i \in \set{1, \ldots, l}$.

\paragraph{Endpoint agreement distance.}
Let $H$ be a MEAG of $N$ and $N'$.
Let $s$ be the number of sprouts of agreement subgraphs of $H$ 
and let $l$ be the number of disagreement edges of $H$.
We define the \emph{endpoint agreement distance} (EAD), denoted by $\dEAD$, of $N$ and $N'$ as 
    $$\dEAD(N, N') = s + l \text{.}$$
Following \citet{WM18} we use their replug operation to show that the EAD is a metric.

\paragraph{Replug distance.}
We define a \emph{replug network} $M$ on $X$ as an undirected multigraph 
such that the leaves and singletons are bijectively labelled with $X$ and all non-leaf vertices have degree three. 
Unlike for a phylogenetic network, $M$ may contain loops and be disconnected.
Let $\replugnets$ be the set of all replug networks on $X$.
Note that $\unets \subseteq \replugnets$.

Let $M \in \replugnets$. 
A \emph{replug operation} is the rearrangement operation that transforms $M$ into a replug network $M' \in \replugnets$ 
by pruning an edge at one vertex and then regrafting it again or by a vertical operation like a \PRP or a \PRM. 
Unlike for \PR, a replug operation does not have to ensure that the resulting network
is connected or proper.
 
Let $N, N' \in \unets$.
We define the \emph{replug distance} $\distR$ of $N$ and $N'$ as the distance of $N$ and $N'$ in $\replugnets$ under the replug operation.
Note that since $\unets$ is connected under \PR, it is also connected as subgraph of $\replugnets$ under the replug operation.
Therefore, the replug distance is well defined and a metric.
We now use the replug distance to prove that the endpoint agreement distance is a metric.

\begin{proposition} \label{clm:unets:EADreplug}
The endpoint agreement distance is equivalent to the replug-distance on $\unets$.
\end{proposition}
\begin{proof}
  Let $N, N' \in \unets$.
  We first prove that $\distR(N, N') \geq \dEAD(N, N')$. 
  Let $d = \distR(N, N')$ and let $\sigma = (N = M_0, M_1, \ldots, M_d = N')$ be a shortest replug sequence.
  Suppose that $N$ and $N'$ are in different tiers and that $N'$ is above $N$.
  Note that we may assume that $\sigma$ does not use any \PRM-like operation, 
  as such an operation and the next \PRP-like operation can be replaced with at most two replug operations that prune and regraft the same edge.
  %Note that $\sigma$ does not use any \PRM-like operation, since removing an edge just to add it again later
  %costs more than simply moving it where it will be in $N'$. The latter is possible since intermediate networks in $\sigma$ only have to be replug networks. 
  We construct a sequence of graphs $(M_0 = H_0, H_1, \ldots, H_d)$ 
  such that $H_i$ has an agreement embedding into $M_i$ for $i \in \set{0, 1, \ldots, d}$.
  The construction will also ensure that~$H_i$ has an agreement embedding into $H_{i-1}$ (using the right number of disagreement edges)
   and thus by \cref{clm:unets:AE:transitive} also into $M_0 = N$.
  
  Suppose $M_i$ is obtained from $M_{i-1}$ by a horizontal replug operation $\theta$ that prunes the edge $e = \set{u, v}$ at $u$.
  Consider the agreement embedding of $H_{i-1}$ into $M_{i-1}$.
  Let $\bar e = \set{\bar u, \bar v}$ be the edge of $H_{i-1}$ that is mapped to a trail $P = (w_1, w_2, \ldots, w_k)$ containing $e$.
  Assume without loss of generality that $\bar u$ is mapped to $w_1$ and $\bar v$ to $w_k$. 
  (If $w_1 = w_k$, we further assume that $\bar e$ imposed with the directed $(\bar u, \bar v)$ is mapped from~$w_1$ towards $w_k$.)
  Now, if $w_1 = u$, that is, $\bar u$ is mapped to $u$, 
  then we can prune~$\bar e$ at~$\bar u$ in~$H_{i-1}$ (unless~$\bar u$ already is a sprout) to obtain $H_i$. 
  If $\bar u$ already is a sprout, set $H_i = H_{i-1}$.
  The agreement embedding of~$H_i$ into $M_i$ is derived from the agreement embedding of $H_{i-1}$ into $M_{i-1}$ 
  and how $\theta$ regrafts~$e$. 
  If~$w_1 \neq u$, a sprout $\bar x$ of $H_{i-1}$ is mapped to $u$.
  In this case, prune $\bar e$ at $\bar u$ (unless~$\bar u$ already is a sprout) to obtain $H_i$. 
  If~$\bar u$ already is a sprout, set $H_i = H_{i-1}$.
  Then apply an embedding change with respect to $\bar x$ and $\bar u$ to obtain an agreement embedding into $M_{i-1}$.
  Derive an agreement embedding into~$M_i$ as in the previous case.
  Clearly $H_i$ has an agreement embedding into $H_{i-1}$.
  
  Next, suppose $M_i$ is obtained from $M_{i-1}$ by a vertical replug operation that adds the edge $e = \set{u, v}$ by
  subdividing the edges $f$ and $f'$. Obtain $H_i$ from $H_{i-1}$ by adding a disagreement edge
  and obtain an agreement embedding of $H_i$ into $M_i$ by mapping the disagreement edge to $e$.
  
  At the end of the sequence, $H_d$ is an endpoint agreement graph of $N$ and $M_d = N'$.
  Since we added at most $d$ sprouts or disagreement edges, it follows that $\distR(N, N') \geq \dEAD(N, N')$.
  
  We now prove that $\dEAD(N, N') \geq \distR(N, N')$.
  Let $H$ be a maximum endpoint agreement graph of $N$ and $N'$.
  Fix ordered agreement embeddings of $H$ into $N$ and $N'$, i.e., no sprout of an agreement subgraph is attached to a disagreement edge.
  Based on agreement embeddings of $H$ into $N$ and $N'$ it is straightforward to use a replug operation
  for each sprout of $H$ to prune an edge of $N$ (or a resulting network) and regraft it according to the agreement embedding of $H$ into $N'$.
  Lastly, if $N'$ is in a tier above $N$, use a \PRP-like replug operation for each disagreement edge of $H$ 
  to add an edge according to the agreement embedding of $H$ into $N'$.
\end{proof}

\begin{corollary} \label{clm:unets:EAD:metric}
The endpoint agreement distance on $\unets$ is a metric.
\end{corollary}

\citet{WM18} showed that the endpoint agreement distance (or rather the replug distance) does not always equal
the SPR-distance of two trees.
Furthermore, they conjectured that computing the endpoint agreement distance is NP-hard for trees.
This and whether it is NP-hard to compute the endpoint agreement distance of two networks remains open.

\section{Relations of distances}
\label{sec:bounds}

In this section we look at the relations of the metrics induced by MAG, MEAG, TBR, and PR.
We start by comparing the agreement distance with the TBR-distance.
As we have seen in \cref{clm:AD:treeRestriction}, they are equivalent on $\utrees$.
Furthermore, we can make the following observations.

\begin{observation} \label{clm:unets:AD:one}
Let $N, N' \in \unets$. Then $\dAD(N, N') = 1$ if and only if $\dTBR(N, N') = 1$.
\end{observation}

\begin{lemma}\label{clm:unets:AD:displaying}
Let $N, N' \in \unets$ be in tiers $r$ and $r'$, respectively, such that $N'$ displays $N$. Let $l = r' - r$.\\
Then $\dAD(N, N') = \dTBR(N, N') = l$.
\end{lemma}
\begin{proof}
The second equality follows from Corollary 5.6 by \citet{JK19}.
The equality also implies that there is a \TBRP-sequence $\sigma$ of length $l$ from $N$ to $N'$.
Let $G$ be the graph obtained from $N$ by adding $l$ disagreement edges.
Then $G$ without its disagreement edges has an agreement embedding into $N$ 
and we can obtain an agreement embedding into $N'$ from $\sigma$ straightforwardly. 
Hence, $G$ is a MAG of $N$ and $N'$, which proves the first equality.
\end{proof}

\pagebreak[4]
\begin{lemma}\label{clm:unets:AD:treeNetwork}
Let $T \in \utrees$ and $N \in \unetsr$. 
Then $\dAD(T, N) = \dTBR(T, N)$.
\end{lemma}
\begin{proof}
	\citet[Theorem 4.13]{JK19} showed that there is tree $T'$ that is displayed by $N$ 
	such that $\dTBR(T,N) = \dTBR(T,T') + \dTBR(T',N)$.
	The tree $T'$ is thus a tree that minimises the TBR-distance to $T$ among all trees displayed by $N$.
	From \cref{clm:AD:treeRestriction} and \cref{clm:unets:AD:displaying} we thus get that $$\dAD(T, N) \leq \dAD(T, T') + \dAD(T', N) = \dTBR(T,T') + \dTBR(T',N) = \dTBR(T,N)\text{.}$$
	For the converse direction, consider a maximum agreement graph $G$ of $T$ and $N$ with $k$ disagreement edges.
	From an ordered agreement embedding of $G$ into $N$, we get that $G$ with $k - r$ disagreement edges embeds onto a tree $T'$ displayed by $N$.
	Hence, $$\dAD(T, N) \geq \dAD(T, T') + \dAD(T, N) = \dTBR(T, N)\text{.}$$  
\end{proof}

After these three cases, where the agreement distance and the TBR-distance are equivalent, 
we show with the following example that this is in general not the case.

\begin{figure}[htb]
  \centering
  \includegraphics{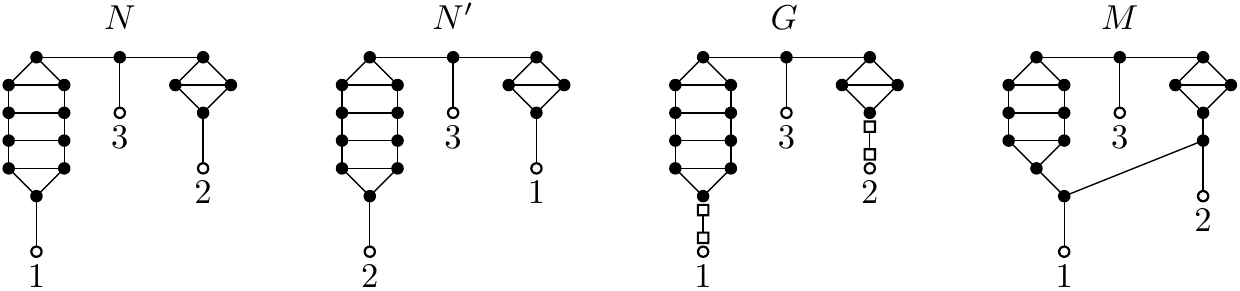}
  \caption{Two networks $N, N' \in \unets$ with $\dAD(N, N') = 2$, but with $\dTBR(N, N') = 3$ as proven in \cref{clm:ADunequalTBR} (for example with a TBR-sequence via $M$).
  The graph $G$ illustrated with an agreement embedding into $N$ is a MAG of $N$ and $N'$.}
  \label{fig:unets:ADunequalTBR}
\end{figure}

\begin{lemma}\label{clm:ADunequalTBR}
The networks $N$ and $N'$ in \cref{fig:unets:ADunequalTBR} have $\dAD(N, N') = 2$, $\dEAD(N, N') = 2$, and $\dTBR(N, N') = 3$.
\end{lemma}
\begin{proof}
Concerning the agreement distance, observe that $\dAD(N, N') > 1$.
Next, note that the graph $G$ in \cref{fig:unets:ADunequalTBR} has agreement embeddings into $N$ and $N'$.
This also yields an agreement embedding of $G$ into~$N'$ by swapping the singletons labelled 1 and 2.
Hence, $G$ with two disagreement edges is a MAG of $N$ and $N'$, which proves that $\dAD(N, N') = 2$.

Concerning the endpoint agreement distance, we see that the leaves 1 and 2 can be swapped with two replug operations.

Concerning the TBR-distance, observe that there is no length two \TBRZ-sequence from $N$ to $N'$.
This can be seen as with only two \TBRZ the leaves 1 and 2 cannot be swapped nor can the two biconnected components be transformed into each other within $\unets$.
To see that $\dTBR(N, N') = 3$, note that $\dTBR(N, M) = 1$ and
that the leaves 1 and 2 can be swapped with a single \TBRZ in $M$ resulting in a network $M'$ with $\dTBR(M', N') = 1$.
\end{proof}

Next, we show that the agreement distance provides a lower and an upper bound on the TBR-distance
of any two networks $N$ and $N'$.

\begin{lemma} \label{clm:unets:AD:TBRrelation:lower}
Let $N, N' \in \unets$.
Then $$\dAD(N, N') \leq \dTBR(N, N') \text{.}$$
\end{lemma}
\begin{proof}
Let $d = \dTBR(N, N')$ and let $\sigma = (N = M_0, M_1, \ldots, M_d = N')$ be a \TBR-sequence from $N$ to $N'$.
To prove the lemma, we show how to obtain an agreement graph $G$ of $N$ and $N'$ with at most $d$ disagreement edges from $\sigma$.

We construct a sequence of graphs $(N = G_0, G_1, \ldots, G_d = G)$ 
such that $G_i$ is an agreement graph of $M_0$ and $M_i$ for $i \in \set{0, 1, \ldots, d}$. 
This holds trivially for $i = 0$. % and we can fix an agreement embedding of $G_0$ into $M_0$.
In the following, when we consider agreement embeddings of $G_i$ into $M_0$ and $M_i$ where $M_0$ and $M_i$ are in different tiers 
then we ignore, for simplicity, that one of the embeddings needs less disagreement edges.

Suppose $M_i$ is obtained from $M_{i-1}$ by a \TBRZ that moves the edge $e = \set{u, v}$.
Let $\bar e$ be the edge of~$G_{i-1}$ that is mapped to a path $P$ that contains $e$ 
by the agreement embedding of $G_{i-1}$ into $M_{i-1}$. 
We distinguish four cases, namely whether $\bar e$ is part of an agreement subgraph and whether $P$ contains only~$e$.
(They are comparable to the cases in the proof of \cref{clm:AD:metric}; see also \cref{fig:unets:ADmetric} again.)
\begin{enumerate}
  \item Assume that $\bar e$ is part of an agreement subgraph and mapped precisely to $e$.
Then obtain $G_i$ by removing $\bar e$ and adding a disagreement edge.
Clearly $G_i$ has an agreement embedding into $N$ and~$M_i$.
  \item Assume that $\bar e$ is part of an agreement subgraph and $P$ has length at least two.
Further assume without loss of generality that neither $u$ nor $v$ is an end vertex of $P$.
Then there there are sprouts $\bar u$ and $\bar v$ that are attached to $\bar e$ in $M_{i-1}$ 
and that are mapped to $u$ and $v$, respectively. 
Again obtain $G_i$ by removing $\bar e$ and adding a disagreement edge $\set{\bar x, \bar y}$.
However, for an agreement embedding of $G_i$ into $M_{i-1}$ map $\set{\bar x, \bar y}$ to $P$ 
and then apply embedding changes with respect to $\bar x$ and $\bar u$ and with respect to $\bar y$ and $\bar v$.
We can derive from this an agreement embedding of $G_i$ into $M_i$.
\item Assume that $\bar e$ is a disagreement edge and mapped precisely to $e$.
Then set $G_i = G_{i-1}$ and it is straightforward to obtain agreement embeddings.
\item Last, assume that $\bar e$ is a disagreement edge and $P$ has length at least two.
There are then again without loss of generality two sprouts attached to $\bar e$.
Set $G_i = G_{i-1}$ and obtain an agreement embedding of $G_i$ into $M_{i-1}$ (and $M_i$) by applying embedding changes as in the second case.
Hence, in either case, we obtain an agreement graph $G_i$ of $M_0$ an $M_i$.
\end{enumerate}
 
Next, suppose $M_i$ is obtained from $M_{i-1}$ by a \TBRM that removes the edge $e = \set{u, v}$.
Like for a \TBRZ, if an edge $\bar e$ of an agreement graph is mapped to $e$,
we obtain $G_i$ from $G_{i-1}$ by removing~$\bar e$. Otherwise we set $G_i = G_{i-1}$.
Furthermore, if $M_{i-1}$ is in a higher tier than $M_0$ we also remove a disagreement edge.
Using again embedding changes if sprouts were attached to $\bar e$, 
it is straightforward to construct an agreement embedding of $G_i$ into $M_i$. 
Thus $G$ is an agreement graph of $M_0$ and $M_i$.

Lastly, suppose that $M_i$ is obtained from $M_{i-1}$ by a \TBRP.
If $M_{i}$ is in a higher tier than $M_0$, then obtain $G_i$ from $G_{i-1}$ by adding a disagreement edge.
Otherwise, set $G_i = G_{i-1}$. In either case, it is clear that $G_i$ is an agreement graph of $M_0$ and $M_i$.

Note that for each \TBR of $\sigma$ we added at most one disagreement edge.
Hence, $G_d$, which is an agreement graph of $M_0$ and $M_d$, is an agreement graph of $N$ and $N'$ with at most $d$ disagreement edges.
This concludes the proof.
\end{proof}

\begin{lemma} \label{clm:unets:AD:TBRrelation:upper}
Let $N, N' \in \unets$.
Then $\dTBR(N, N') \leq 2 \dAD(N, N') \text{.}$
\end{lemma}
\begin{proof}
Suppose $N$ and $N'$ are in tier $r$ and $r'$, respectively, and that $r' \geq r$. Let $l = r' - r$.
Let $d = \dAD(N, N')$ and $k = d - l$.
Let $G$ be a MAG of $N$ and $N'$ with agreement subgraphs $S_1, \ldots, S_m$ 
and disagreement edges $E_1, \ldots, E_{k}, E_{k+1}  \ldots, E_{d}$.
To prove the theorem, we construct a \TBR-sequence 
    $$\sigma = (N = M_0, M_1, \ldots, M_d, \ldots, M_{d + k} = N')$$
  such that $M_{i}$ is obtained from $M_{i-1}$ by a \TBRP for $i \in \set{1, \ldots, d}$ and by a \TBRM for $i \in \set{d + 1, \ldots, d + k}$.  
Along $\sigma$ we maintain a series of graphs $G_0, \ldots, G_{d + k}$ 
    such that $G_i$ has an agreement embedding into~$M_i$.

Fix ordered agreement embeddings of $G$ into $N$ and $N'$, which is possible by \cref{clm:unets:orderedAE}.
Let $N_1, \ldots, N_m$ be the subgraphs of $N$ to which $S_1, \ldots, S_m$ of $G$ are mapped, respectively.
We define $N'_1, \ldots, N'_m$ analogously for $N'$.
Note that the disagreement edges of $G$ are mapped to paths in $N$ (resp.~$N'$) that (as a whole) pairwise connect the $N_i$'s (resp. $N'_i$'s).  
The idea is now as follows.
From $M_0$ to $M_d$ we add $d$ edges to reconstruct the paths that connect the $N'_i$'s as in $N'$ while maintaining the paths that connect the $N_i$'s as in $N$.
From $M_d$ to $M_{d + k}$ we then remove edges guided by how the paths connect the $N_i$'s in $N$.
This is illustrated in \cref{fig:unets:AD:TBRrealation:upper:idea}.
We now define the graphs $G_i$ formally and explain how to construct $\sigma$.

\begin{figure}[htb]
  \centering
  \includegraphics{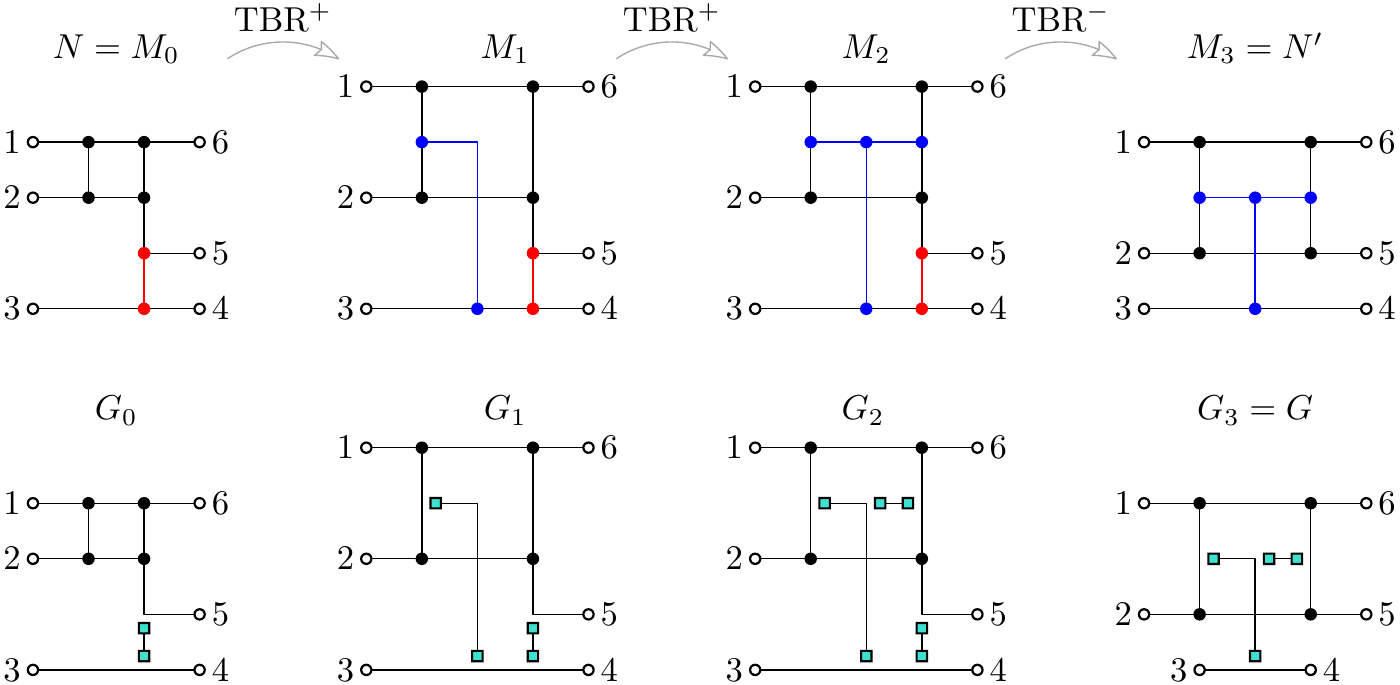}
  \caption{Construction of a TBR-sequence from $N$ to $N'$ based on a MAG $G$ of $N$ and $N'$.}
  \label{fig:unets:AD:TBRrealation:upper:idea}
\end{figure}

Let $G_0$ be $G$ without $l$ disagreement edges. 
Therefore, $G_0$ has $k$ disagreement edges and an agreement embedding into $N_0 = N$ without spare disagreement edges.
For $i \in \set{1, \ldots, d}$ let $G_i$ be $G_{i-1}$ plus one disagreement edge.
Next, for $i \in \set{d+1, \ldots, d+k}$ let $G_i$ be $G_{i-1}$ minus one disagreement edge.
Let~$E_j^i$ for $i \in \set{0, 1, \ldots, d + k}$ and $j \in \set{1, \ldots}$ denote the disagreement edges of $G_i$.

In $M_0$ colour the subgraph to which agreement subgraphs of $G_0$ are mapped black.
Colour all other vertices and edges red.
Obtain $M_1$ from $M_0$ as follows.
First, assume that $E_1^{d+k}$ is attached to edges $\bar e$ and $\bar e'$ of agreement subgraphs in $M_{d+k}$.
Consider the paths $P_{\bar e}$ and $P_{\bar e'}$ in $N_0$ to which $\bar e$ and $\bar e'$ are mapped. 
Ignoring vertices on these paths that are incident to red edges, we can perceive $P_{\bar e}$ and $P_{\bar e'}$ as edges $e$ and $e'$.
Then add with a \TBRP an edge $f$ from $e$ to $e'$.
Next assume that $E_1^{d+k}$ is attached to a labelled singleton $\bar u$ of $G$ in $M_{d+k}$. 
Then let $u$ be the leaf of $M_0$ to which $\bar u$ is mapped. 
Note that $u$ is incident to a red edge, say $e$, in $M_0$.
Obtain $e'$ as in the previous case. 
If $E_1^{d+k}$ is attached to two labelled singletons, then obtain a second red edge $e'$ analogous to how we obtained $e$. 
In either case, apply the \TBRP that adds an edge $f$ from $e$ to $e'$ .
Let $M_1$ be the resulting network. Colour the new edge $f$ blue.
Obtain an agreement embedding of $G_1$ into $M_1$ by extending the agreement embedding of $G_0$ into $N_0$ by mapping $E_{k + 1}^1$ to $f$.
Note that $M_1$ is a proper phylogenetic network since adding an edge (with a \TBRP) to a proper network yields a proper network.
In particular, edges obtained from subdividing $e$ and~$e'$ still lie on paths between leaves and so does thus $f$. 
Repeat this process to obtain $M_i$ from $M_{i-1}$ for $i \in \set{2, \ldots, d}$ based on how $E_i^{d+k}$ embeds into $M_{d+k}$. 

Observe that $M_d$ and $G_d$ with its agreement embedding into $M_d$ can also be obtained by applying the construction we used to obtain $M_d$ from $M_0$
by starting from $M_{d+k}$ and considering the agreement embedding of $G_0$ into $M_0$ (instead of the agreement embedding of $G_{d+k}$ into $M_{d+k}$).
The only difference is that in the two resulting agreement embeddings of $G_d$ into $M_d$ blue disagreement edges might be attached to red edges or vice versa,
wherever there is a labelled singleton (leaf). 
Nevertheless, this shows that we can construct the full TBR-sequence $\sigma$.
To conclude the proof, note that $d + k \leq 2d$.
\end{proof}

From \cref{clm:unets:AD:TBRrelation:lower} and \cref{clm:unets:AD:TBRrelation:upper} we get the following theorem.

\begin{theorem} \label{clm:unets:AD:TBRrelation}
Let $N, N' \in \unets$.
Then $$\dAD(N, N') \leq \dTBR(N, N') \leq 2 \dAD(N, N') \text{.}$$
\end{theorem}

\citet[Corollary 3.3]{JK19} showed that the PR-distance is bound from below by the \TBR-distance
and from above by at most twice the \TBR-distance. Hence, we get the following corollary.

\begin{corollary} \label{clm:unets:AD:PRrelation}
Let $N, N' \in \unets$.
Then $$\dAD(N, N') \leq \dPR(N, N') \leq 4 \dAD(N, N') \text{.}$$
\end{corollary}

% > EAD ---------------------------------------------------------------
We now turn to the endpoint agreement distance and look at its relation to the agreement distance and the PR-distance.

\begin{proposition} \label{clm:unets:ADvsEAD}
Let $N, N' \in \unets$.
Then $$\dAD(N, N') \leq \dEAD(N, N') \leq 2\dAD(N, N')\text{.}$$
\end{proposition}
\begin{proof}
  We start with the first inequality.
  Let $H$ be a maximum endpoint agreement graph of $N$ and $N'$.
  Suppose $H$ has $s$ sprouts in agreement subgraphs and $l$ disagreement edges.
  We prove that there is an agreement graph $G$ of $N$ and $N'$ with at most $s + l$ disagreement edges.
  For this, we construct a sequence of graphs $(G_s, G_{s-1}, \ldots, G_0)$
  such that each $G_{i}$ is an endpoint agreement graph of $N$ and $N'$ 
  with at most $i$ sprouts in agreement subgraphs and at most $l + (s-i)$ disagreement edges.
  (We slightly abuse the definition of disagreement edges here 
  and consider any edge incident with two sprouts of $G_i$ as a disagreement edge.)
  Therefore, setting $G = G_0$ will prove the proposition.
  
  Suppose $\bar u$ is a sprout of an agreement subgraph of $G_{i+1}$.
  Let $\bar e = \set{\bar u, \bar v}$ be the edge incident to $\bar u$.
  If $\bar v$ is also a sprout, set $G_i = G_{i+1}$ and classify $\bar e$ as a disagreement edge. 
  Furthermore, we can also set $G_{i-1} = G_i$ since we eliminated two sprouts of agreement subgraphs at once.
  Otherwise, obtain $G_i$ from $G_{i+1}$ by pruning $\bar e$ from $\bar v$. 
  Since $\bar v$ is either a degree vertex or a labelled leaf, 
  we can directly derive agreement embeddings of $G_i$ into $N$ and $N'$ 
  from the agreement embeddings of $G_{i+1}$.
  Since every step reduces the number of sprouts in agreement subgraphs by at least one
  and adds at most one disagreement edge, $G_0$ is as desired.
  
  For the second inequality, note that a MAG of $N$ and $N'$ with $k$ disagreement edges
  is also an EAG of $N$ and $N'$ with $2k$ sprouts.
\end{proof}

\begin{theorem} \label{clm:unets:EAD:PRrelation}
Let $N, N' \in \unets$.
Then $$\dEAD(N, N') \leq \dPR(N, N') \leq 3 \dEAD(N, N') \text{.}$$
\end{theorem}
\begin{proof}
  Without loss of generality, assume that $N$ is not in a higher tier than $N'$. 
  For the lower bound, consider a shortest \PR-sequence $\sigma$ from $N$ to $N'$.
  Note that $\sigma$ is also a replug sequence. 
  There is thus a replug sequence from $N$ to $N'$ whose length is at most the length of $\sigma$.
  The lower bound now follows from \cref{clm:unets:EADreplug}.
  
  Next, we prove the upper bound.
  Let $H$ be a MEAG for $N$ and $N'$ with $s$ sprouts in agreement subgraphs and $l$ disagreement edges.
  Fix ordered endpoint agreement embeddings of $H$ into $N$ and $N'$.
  Let $d = \dEAD(N, N') = s + l$.
  We construct a PR-sequence $\sigma = (N = M_0, M_1, \ldots, M_{d'} = N')$ with $d' \leq 3d$.
  Along $\sigma$, we maintain a sequence of graphs $(H = H_0, H_1, \ldots, H_{d'})$ that consist of $H$ plus possibly extra disagreement edges such
  that $H_i$ has an agreement embedding into $M_i$. We call these extra disagreement edges \emph{ghost disagreement edges}.
  
  Let $E_1, \ldots, E_l$ be the disagreement edges of $H$.
  For $i \in \set{1, \ldots, l}$ obtain $M_i$ from $M_{i-1}$ by adding an edge $e$ with a \PRP
  according to where the disagreement edge $E_i$ is attached to in the agreement embedding of $H$ into $N'$.
  If $E_i$ is attached to a labelled singleton $\bar v$, then both $\bar v$ and a sprout $\bar u$ are mapped to a leaf $v$ of $M_{i-1}$.
  In this case, attach $e$ to the edge incident to $v$.
  For an agreement embedding of $H$ into $M_i$, map the disagreement edge $E_i$ to the newly added edge.
  If $E_i$ should be attached to $v$, apply the appropriate embedding change with $\bar u$. 
  Set $H_i = H$.
  
  We now use \PRZ to move edges according from where sprouts are attached to in $M_i$ to where they are attached to in $N'$.
  If we have done this for a sprout, we call it \emph{handled} and \emph{unhandled} otherwise.
  Let $\bar u$ be an unhandled sprout of $H_i$. 
  Let $\bar u$ and its incident edge be mapped to $u$ and $e = \set{u, v}$ of $M_{i-1}$, respectively.
  If $e$ can be pruned at $u$ and attached to the edge according to where $\bar u$ is mapped to in $N'$ 
  such that the result is a proper phylogenetic network, then apply this \PRZ to obtain $M_i$. 
  Set $H_i = H_{i-1}$.
  Note that in the case that $\bar u$ is mapped to a leaf $w$ in $N'$, then $e$ is attached the edge incident to $w$, 
  and we apply the appropriate embedding change for $H_i$. 
  (Apply this to each unhandled sprout where possible).
  Otherwise, use a \PRP to add a (ghost) edge $f$ from $e$ to the edge incident to leaf 1 to obtain $M_i$.
  Obtain $H_i$ from $H_{i-1}$ by adding a ghost disagreement edge $F$. 
  Map $F$ to $f$ and  apply an embedding change with respect to $F$ and $\bar u$. % as illustrated in \cref{fig:unets:EAD:embeddingChange}.
  Note that now the first case applies for $\bar u$ and $e$. Thus we also obtain $M_{i+1}$ and $H_{i+1}= H_i$.
  When all sprouts are handled, the agreement embedding of $H_i$ without ghost disagreement edges is mapped to the subgraph of $M_i$
  that is precisely a subdivision of $N'$. We thus need at most $s$ further \PRM to remove all ghost edges.
  
  In total, this process requires at most $l + 3s = d' \leq 3d$ \PR. This proves the upper bound.
\end{proof}

\section{Concluding remarks}
\label{sec:discussion}

In this paper, we defined maximum agreement graphs (MAG) 
for two unrooted, proper, binary phylogenetic networks. 
Like maximum agreement forests for trees, a MAG models how two networks
agree on subgraphs that stay untouched when moving edges with TBR operations. 
If the two networks are in different tiers, then a MAG also models how the networks disagree on that.
Based on MAGs, we defined the agreement distance of phylogenetic networks.
By showing that this new metric is equivalent to the TBR-distance for two trees,
we obtained that it is NP-hard to compute the agreement distance. 

We have seen that the agreement distance and the TBR-distance are equivalent for trees and for networks with distances of at most one. 
Furthermore, we know that the agreement distance of a tree and a network equals their TBR-distance. 
On the other hand, there are networks $N$ and $N'$, as in \cref{fig:unets:ADunequalTBR}, with agreement distance two but higher TBR-distance.
However, note that $N$ and $N'$ are in tier seven. 
It is therefore of interest to further study when exactly the agreement distance is equivalent to the TBR-distance and when not.
In general, we showed that the agreement distance of two networks provides a natural lower bound
and an upper bound with factor two on their TBR-distance.
If we drop the requirement that networks have to be proper, it is also open whether the agreement distance 
and the TBR-distance are equivalent or not.

Like SPR on trees has been generalised to PR on networks, we have generalised maximum endpoint agreement forests of \citet{WM18}
to maximum endpoint agreement graphs (MAEGs) for networks.
We showed that MAEGs induce a metric, called endpoint agreement distance, which bounds the PR-distance naturally from below 
and with a factor of three from above.
Furthermore, we showed that the agreement distance provides bounds on the PR-distance either via the TBR-distance
or via its relation to the endpoint agreement distance.

MAFs and MAEFs have been used to develop algorithms that compute the TBR-distance and PR-distance of two trees, respectively.
It is thus of interest to see whether MAGs can be utilised to develop approximation algorithms for the agreement distance.
Note that such an algorithm would also be an approximation algorithm of the TBR- and the PR-distance.

\phantomsection 
\pdfbookmark[1]{Acknowledgements}{acknowledgements} 
\acknowledgements
The author would like to thank the New Zealand Marsden Fund for their financial support
and the anonymous reviewers for their very helpful comments.

\phantomsection
\pdfbookmark[1]{References}{references}
\nocite{*}
\bibliographystyle{abbrvnat} 
% \bibliography{sources}  

\end{document}